\documentclass[A4]{article}
\usepackage[pdftex]{hyperref}
\usepackage{titlesec}
\usepackage{mathrsfs}
\usepackage{amsfonts}
\usepackage{latexsym}
\usepackage{epsfig}
\usepackage{indentfirst}
\graphicspath{{figures/}}
\usepackage{amsmath}
\usepackage{amsthm}
\usepackage{graphicx}
\usepackage{amssymb}
\usepackage{amscd}
\usepackage{latexsym}
\usepackage{subfigure}
\usepackage{epstopdf}
\usepackage{color}
\DeclareGraphicsExtensions{.eps}
\input xy
\xyoption{all}

\theoremstyle{definition}
\newtheorem{thm}{\textbf{Theorem} }[section]
\newtheorem{conj}{\textbf{Conjecture}}

\newtheorem{prop}[thm]{\textbf{Proposition}}

\numberwithin{equation}{section} \makeatletter
\renewenvironment{proof}[1][\proofname]{\par
    \pushQED{\qed}%
    \normalfont \topsep6\p@\@plus6\p@ \labelsep1em\relax
    \trivlist
    \item[\hskip\labelsep
        \bfseries #1]\ignorespaces
}{%
    \popQED\endtrivlist\@endpefalse
} \makeatother
\renewcommand{\proofname}{Proof}

\title{\bfseries \large{A note on $0$-bipolar knots of concordance order two}}
\author{ \normalsize{Wenzhao Chen}\footnote{chenwenz@math.msu.edu}}
\date{}

\titleformat{\section}
  {\normalfont\large\bfseries}
  {\thesection}{1em}{}
  
\begin{document}
\maketitle
% insert the table of contents
%\tableofcontents
%\newpage
\parindent = 18pt
%\begin{abstract}
%
\begin{abstract}
\setlength{\parindent}{1pt} \setlength{\parskip}{1.5ex plus 0.5ex
minus 0.2ex} \noindent
Let $\mathcal{T}$ be the group of smooth concordance classes of topologically slice knots, and $\{0\}\subset\cdots\subset \mathcal{T}_{n+1}\subset\mathcal{T}_{n}\subset \cdots\subset \mathcal{T}_{0}\subset \mathcal{T}$ be the bipolar filtration. In this paper, we show that a proper collection of the knots employed by Hedden, Kim, and Livingston to prove $\mathbb{Z}_2^{\infty} < \mathcal{T}$ can be used to see $\mathbb{Z}_2^{\infty} < \mathcal{T}_0/\mathcal{T}_1$.
\end{abstract}

\section{Introduction}
Viewing the 3-sphere as the boundary of the 4-ball, a knot $K\subset S^3$ is said to be \textit{smoothly slice} if it bounds a smoothly embedded disk in $D^4$. An analogous notion called \textit{topological sliceness} is defined if one allows the aforementioned disk to be topologically embeded and locally flat. These two notions of sliceness are different, manifesting in the nontriviality of the group of smooth concordance classes of topologically slice knots $\mathcal{T}$. As an attempt to understand finer structure of the gap between the smooth and topological category, Cochran, Harvey, and Horn introduced the so-called bipolar filtration on $\mathcal{T}$:$$\{0\}\subset\cdots\subset \mathcal{T}_{n+1}\subset\mathcal{T}_{n}\subset \cdots\subset \mathcal{T}_{0}\subset \mathcal{T},$$ where $\mathcal{T}_n$ is the subgroup generated by $n$-bipolar topologically slice knots \cite{CHH13}. Here, a knot is said to be \emph{$n$-bipolar} if it is both \emph{$n$-positive} and \emph{$n$-negative}, which we recall below.

A knot $K$ (not necessarily topologically slice) is said to be \emph{$n$-positive} (resp.\ \emph{$n$-negative}) if it bounds a smoothly embedded disk $\Delta$ in a four-manifold $V$ such that
\begin{enumerate}
\item $\pi_1(V)=0$;
\item $V$ has positive definite (resp.\ negative definite) intersection form; 
\item $H_2(V)$ has a basis represented by a collection of surfaces $\{S_i\}$ disjointly embedded in $V\setminus \Delta$ such that $\pi_1(S_i)\subset \pi_1(V\setminus\Delta)^{(n)}$ for each $i$, where $\pi_1(V\setminus\Delta)^{(n)}$ denotes the $n$-th derived subgroup of $\pi_1(V\setminus\Delta)$.
\end{enumerate}

Loosely speaking, the third condition in the above definition implies that the deeper a knot $K$ lies in the filtration, the simpler the corresponding surfaces $\{S_i\}$ are. It is better to appreciate this condition by considering an extreme case. Note if the $S_i$'s were actually spheres, then by the first two conditions and the diagonalization theorem of Donaldson \cite{Don83}, one can see that the $S_i$'s have self-intersection $\pm{1}$ and hence can be repeatedly blown down, achieving a homotopy 4-ball in which the knot bounds a disk. 

One natural theme regarding the bipolar filtration is to understand its nontriviality. In the paper where the filtration is introduced, Cochran, Harvey, and Horn showed $\mathbb{Z}^{\infty}<\mathcal{T}/\mathcal{T}_0$ and $\mathcal{T}_1/\mathcal{T}_2$ has positive rank \cite{CHH13}. Later, Cochran and Horn showed $\mathbb{Z}^{\infty}<\mathcal{T}_0/\mathcal{T}_1$ \cite{CH15}. The nontriviality of deep levels of the filtration is recently proved by Cha and Kim, who showed $\mathbb{Z}^{\infty}<\mathcal{T}_i/\mathcal{T}_{i+1}$ for all $i\geq 2$ \cite{CK17}.  This note aims at enriching the list. We observe that the knots studied by Hedden, Kim, and Livingston can be used to show that $\mathcal{T}_0/\mathcal{T}_1$ also contains a infinite subgroup generated by 2-torsion elements.
\begin{thm}\label{main}
$\mathbb{Z}_2^{\infty} < \mathcal{T}_0/\mathcal{T}_1$.
\end{thm}
It is natural to believe similar results should be true for successive quotients of deeper levels.
\begin{conj}
$\mathbb{Z}_2^{\infty} < \mathcal{T}_i/\mathcal{T}_{i+1}$ for all $i\in \mathbb{N}$. 
\end{conj}
In fact, if one does not restrict the attention to topologically slice knots, Cochran, Harvey and Horn already showed the existence of infinitely many linearly independent amphichiral $n$-bipolar knots for $n\neq 1$ (cf.\ Theorem 7.1 of \cite{CHH13}).

The organization of the rest of paper goes as follows. In Section 2 we recall the construction of the knots used in \cite{HKL16}, and show that a proper collection of the knots is 0-bipolar. In Section 3 we briefly recall necessary facts from Heegaard Floer homology as a preparation for the proving the main theorem. In Section 4 we give a proof for Theorem \ref{main}.

\section*{Acknowledgment} I am grateful for Matt Hedden for his ongoing support and interesting conversations that led to this question. This work also benefited greatly from Min Hoon Kim's extremely detailed and valuable feedback. I also thank Chuck Livingston for his interest and useful comments. 
\begin{figure}[!ht] \label{fig1}
\centering{
\resizebox{60mm}{!}{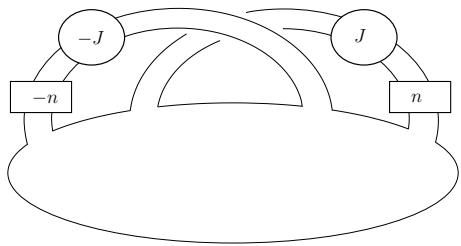}
\caption{$K_{J,n}$. Here the box with integer $n$ (resp.\ $-n$) stands for $n$ full right-handed (resp.\ left-handed) twists}
}
\end{figure}
\section{The knots $K_{n,k}$}
We begin by recalling the construction of the amphichiral knots $K_{n,k}$ considered in \cite{HKL16}. First, given any knot $J$ and an integer $n$, one can construct $K_{J,n}$ by infection and twisting as in Figure \ref{fig1}. One then define $K_{n,k}$ to be $K_{D_k,n}\#K_{U,n}$.
Here, $D_k$ refers to the connected sum of $k$ copies of $D=Wh^+(T_{2,3})$, the untwisted positive whitehead double of the right-hand trefoil, and $U$ denotes the unknot.

Note $K_{J,n}$ is negatively amphicheiral since it is the boundary of a surface obtained by plumbing two bands, with one band being the mirror of the other one. Since $D_k$ is topologically slice \cite{Fre82}, $K_{D_k,n}$ is topologically concordant to $K_{U,n}=-K_{U,n}$, implying $K_{D_k,n}\#K_{U,n}$ is topologically slice. In summary, $K_{n,k}$ is a topologically slice negative amphicheiral knot. \\

The rest of this section is devoted to proving the next proposition,  which shows $K_{n,k}$ are $0$-bipolar when $n$ and $k$ are properly chosen.
\begin{prop}\label{K_n,k are bipolar}
If $n\geq 4k$, then $K_{n,k} \in \mathcal{T}_0$.
\end{prop}

\begin{figure}[!ht]\label{fig2}
\centering{
\resizebox{116mm}{!}{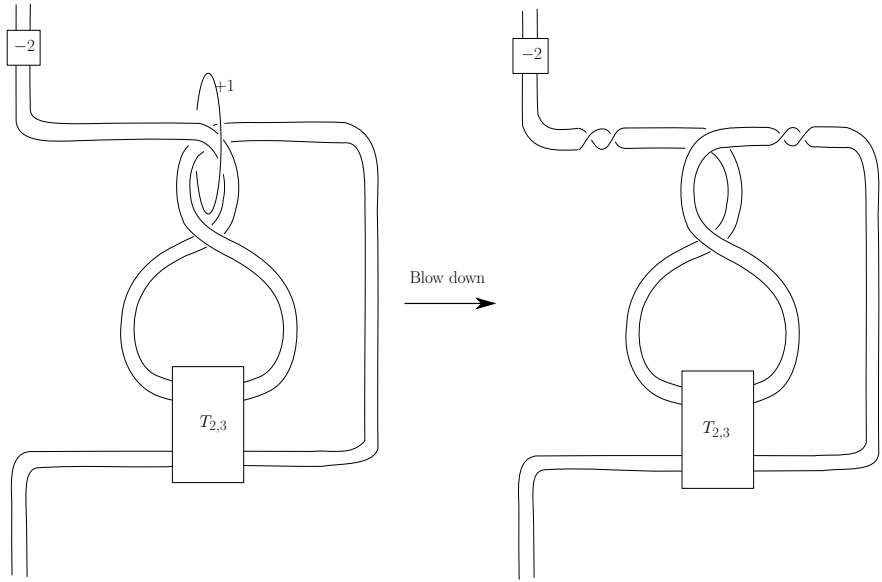}
\caption{The blow-down operation used to untie a band infected by $D$: note the writhe of the above diagram for $D$ is $2$, so we introduced $-2$ full twists to balance that.}
}
\end{figure}

Before giving the proof, we recall some standard notations for convenience. The monoid of concordance classes of $n$-positive (resp.\ $n$-negative) knots is usually denoted by $\mathcal{P}_n$ (resp.\ $\mathcal{N}_n$), and we denote the group of concordance classes of $n$-bipolar knots by $\mathcal{B}_n$. 
\begin{proof}
Note the knots $K_{n,k}$ are topologically slice, so it suffices to see both $K_{D_k,n}$ and $K_{U,n}$ are in $\mathcal{B}_0$. Moreover, since both $K_{D_k,n}$ and $K_{U,n}$  are amphicheiral it suffices to show both of them are in $\mathcal{N}_0$. For $K_{U,n}$, one can see it lies in $\mathcal{N}_0$ since it can be unknotted by changing negative crossings. To prove $K_{D_k,n}\in \mathcal{N}_0$, note $D=Wh^+(T_{2,3})$ can be unknotted by changing a single positive crossing. As in \cite{CH15}, we can untie the band infected by $D$ by blowing down a $+1$-framed unknot around the corresponding crossing. See Figure \ref{fig2}. However, other than changing the crossing, this operation also produces band twisting. Using Figure \ref{fig2} as a local model, it is not hard to see that after carrying out this blow-down operation $k$-times, one for each copy of $D$ on the right-hand band in Figure \ref{fig1}, the resulting knot is $J_{k,n}$ as shown in Figure 3. Note the right-hand band in Figure \ref{fig3} has $(n-4k)$ full twists, inducing $2(n-4k)$ negative crossings when $n\geq 4k$. We can unknot $J_{k,n}$ by changing $(n-4k)$ of these negative crossings to positive, which can be achieved by the typical way of blowing down $+1$-framed unknots around the crossings. In summary, $K_{D_k,n}$ can be unknotted by blowing down $(n-3k)$ copies of $+1$-framed unknots, all of which have linking number $0$ with $K_{D_k,n}$. Hence $K_{D_k,n}$ bounds a disk $\Delta$ in $W$ where $W$ is punctured $\overline{\mathbb{CP}}^{n-3k}$ such that $[\Delta,\partial \Delta]=0\in H_2(W,\partial W)$. Therefore, $K_{D_k,n}\in \mathcal{N}_0$.
\end{proof}
\begin{figure}[!ht]\label{fig3}
\centering{
\resizebox{60mm}{!}{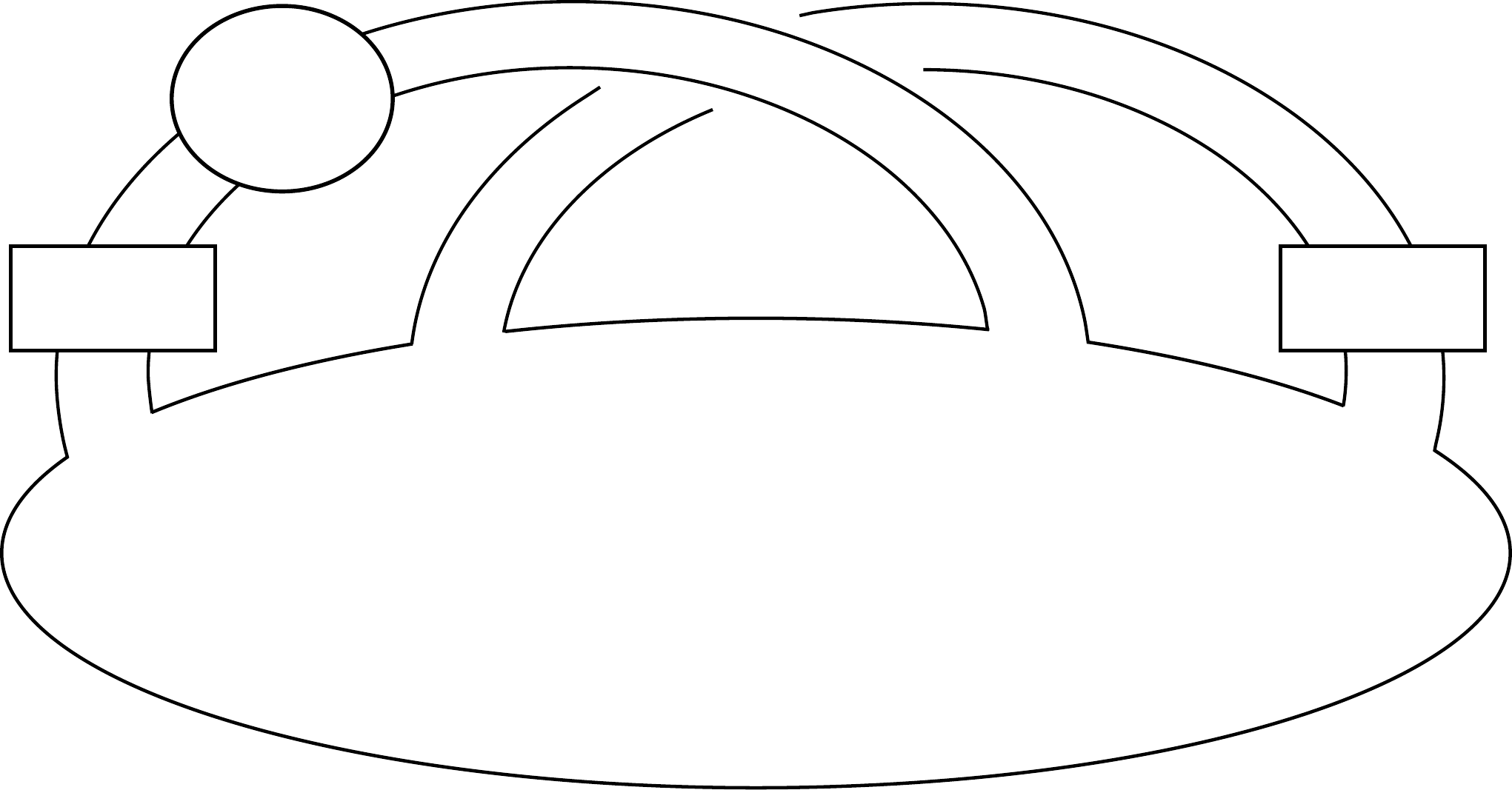}
\caption{$J_{k,n}$}
\label{fig3}
}
\end{figure}
\section{Obstruction from Heegaard floer homology}
\subsection{$d$-invariant and bipolar filtration}
The obstruction we use to see the knots $K_{n,k}$ not lying in $\mathcal{T}_1$ is provided by the \textit{correction term} derived from Heegaard Floer homology, which is often called the $d$-invariant as well. Introduced by Oszv\'{a}th and Szab\'{o}, the correction term $d(Y,\mathfrak{s})$ is a rational number associated to a rational homology sphere $Y$ equiped with a $Spin^c$-structure $\mathfrak{s}$. A precise definition for $d$-invariant will not be needed for our purpose, so we refer the interested reader to \cite{OS03} for a more detailed exposition. Instead, we will gather a few facts below with an eye towards our application. \\

The first thing we recall is the additivity of the $d$-invariant.
\begin{thm}[\cite{OS03}]
$d(Y_1\#Y_2,\mathfrak{s}_1\#\mathfrak{s}_2)=d(Y_1,\mathfrak{s}_1)+d(Y_2,\mathfrak{s}_2)$ .
\end{thm}

Secondly we recall a result which relates correction terms and the bipolar filtration. Given a knot $K\subset S^3$, denote by $M(K)$ the two-fold branched cover of $S^3$ along $K$. Note that $M(K)$ is a $\mathbb{Z}_2$-homology sphere and therefore admits a unique spin structure. This allows us to describe the set $\text{spin}^c$ structures on $M(K)$ conveniently: denote by $\mathfrak{s}_0$ the $\text{spin}^c$ structure induced by the unique spin structure, and hence in view of transitivity of the action of $H^2(Y)$ on $\text{spin}^c(M(K))$, any $\text{spin}^c$ structure on $M(K)$ can be denoted as $\mathfrak{s}_0+PD[x]$ for some $x\in H_1(M(K))$, where $PD[\cdot]$ denotes Poincar\'{e} duality. We sometimes abbreviate  $\mathfrak{s}_0+PD[x]$ to $\mathfrak{s}_x$. With these conventions in mind, we have the following theorem, a more general form of which can be found in \cite{CHH13}.
 
\begin{thm}[Theorem 6.2 and Theorem 6.5 of \cite{CHH13}]\label{d invariant and 1-bipolar knots}
Assume $K\in \mathcal{B}_1$ and let $M(K)$ be the two-fold branched cover of $S^3$ along $K$, then there exist metabolizers $G_i<H_1(M(K))$, $i=1,2$ for the $\mathbb{Q}/\mathbb{Z}$-valued linking form on $H_1(Y)$ such that for all $z_i\in G_i$, $i=1,2$ we have
$$
 \begin{aligned}
 d(M(K),\mathfrak{s}_0+PD[z_1])\leq 0,\\
 d(M(K),\mathfrak{s}_0+PD[z_2])\geq 0,
 \end{aligned}
 $$
where $\mathfrak{s}_0$ is the unique $\text{spin}^c$ structure on $Y$ induced by the spin structure. 
\end{thm}

\subsection{$d$-invariants of two-fold branched cover along $K_{n,k}$}
 The proof of the main theorem will be achieved by applying Theorem \ref{d invariant and 1-bipolar knots} to $M(K_{k,n})=M(K_{D_k,n})\#M(K_{U,n})$. Therefore, in this subsection we will gather some facts about $d$-invariants of $M(K_{k,n})$ from \cite{HKL16}. To begin with, we need a description for the first homology group of $M(K_{k,n})$, as well as a compatible description for the metabolizers of the linking form. 

\begin{figure}[!ht]\label{fig4}
\centering{
\resizebox{95mm}{!}{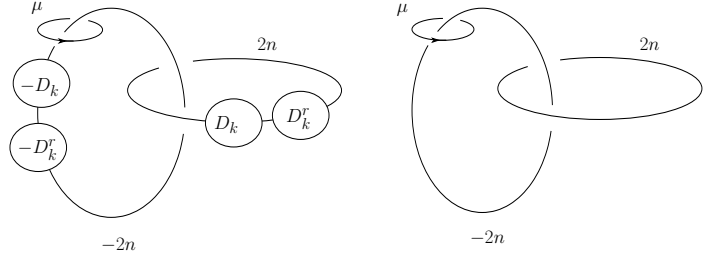}
\caption{Surgery description for $M(K_{D_k,n})$ (left) and $M(K_{U,n})$ (right): the superscript $r$ stands for orientation reversal of the corresponding knot.}
}
\end{figure}

Implementing an algorithm in \cite{AK79}, one can obtain surgery descriptions for $M(K_{D_k,n})$ and $M(K_{U,n})$ as illustrated in Figure 4. The first homology group of both 3-manifolds can be seen to be isomorphic to the cyclic group $\mathbb{Z}_{4n^2+1}$ by a standard computation. In each figure, a meridian curve $\mu$ is labeled in order to serve as a prefered generator of the first homology group of the corresponding 3-manifold.  Summarizing this we have the following.
\begin{prop}[Proposition 2.1 of \cite{HKL16}]
The surgery description in Figure 4 provides an isomorphism $H_1(M(K_{D_k,n}))\cong H_1(M(K_{U,n}))\cong \mathbb{Z}_{4n^2+1}\langle \mu \rangle$.
\end{prop}

The metabolizers for the linking form on $H_1(M(K_{n,k}))$ has the following description.
\begin{prop}[Lemma 3.5 of \cite{HKL16}]\label{metabolizer}
With the isomorphism $H_1(M(K_{n,k}))\cong \mathbb{Z}_{4n^2+1}\oplus\mathbb{Z}_{4n^2+1}$ given by the previous proposition, if $4n^2+1$ is square free, then each metabolizer for the linking form is generated by $(1,b)$ for some $b$ such that $b^2+1\equiv 0 \	(\text{mod}\ 4n^2+1)$.
\end{prop}

Though it was not explicitly stated as a proposition in \cite{HKL16}, the following deep result is the key to the proof of the main theorem of that paper. This technical result is achieved through delicate computations of various knot Floer homology groups and a nice application of the surgery mapping cone formula for Heegaard Floer homology.  We give the summary below and refer the reader to Section 3 of \cite{HKL16} for the details.
\begin{prop}\label{technical}
Fix a positive integer $n$ such that  $4n^2+1$ is square free and is a product of at most two primes, then there are at most four choices for $k$ with $0<k <n/2$ such that there exist an integer $b$ with $b^2+1\equiv 0 \	(\text{mod}\ 4n^2+1)$ and $d(M(K_{D_k,n}), \mathfrak{s}_x)+d(M(K_{U,n}),\mathfrak{s}_{bx})=0$, for all $x\in \mathbb{Z}_{4n^2+1}$.
\end{prop}
\section{Proof of the main theorem}
This section is devoted to proving Theorem \ref{main}. We begin by showing certain $K_{n,k}$ is not in $\mathcal{T}_1$.
\begin{prop}
Let $n>20$ be an integer such that $4n^2+1$ is square free and is a product of at most two primes, then one can choose some $k$, such that $K_{n,k}\in \mathcal{T}_0$ but $K_{n,k}\notin \mathcal{T}_1$.
\end{prop}
\begin{proof}
We claim that if $K_{n,k}\in \mathcal{T}_1$, then there is some $b$ with $b^2+1\equiv 0 \	(\text{mod}\ 4n^2+1)$ and 
\begin{equation}
d(M(K_{D_k,n}), \mathfrak{s}_x)+d(M(K_{U,n}),\mathfrak{s}_{bx})=0, \,\text{for all}\, x\in \mathbb{Z}_{4n^2+1}.
\end{equation}
First we give a proof for the proposition assuming this claim for a moment. Note $n/4>5$ by our choice of $n$, it then follows from Proposition \ref{technical} that we can choose some $k$ with $0<k\leq n/4$ such that Equation (4.1) is violated for any possible metabolizer.  Hence $K_{k,n}\in \mathcal{T}_0$ by Proposition \ref{K_n,k are bipolar} but $K_{k,n}\notin \mathcal{T}_1$ by the claim.\\
Now we move to prove the above claim. Note by Theorem \ref{d invariant and 1-bipolar knots} and Proposition \ref{metabolizer} , there exist $b_i$, $i=1,2$, such that $b_i^2+1\equiv 0 \	(\text{mod}\ 4n^2+1)$ and 
\begin{equation}
d(M(K_{D_k,n}), \mathfrak{s}_x)+d(M(K_{U,n}),\mathfrak{s}_{b_1x})\geq 0,
\end{equation}
and
\begin{equation}
d(M(K_{D_k,n}), \mathfrak{s}_x)+d(M(K_{U,n}),\mathfrak{s}_{b_2x})\leq 0,
\end{equation}
for all $x\in \mathbb{Z}_{4n^2+1}$. Note $b_i$ is relatively prime to $4n^2+1$ since $b_i^2+1\equiv 0 \	(\text{mod}\ 4n^2+1)$. Therefore, from (4.2) we have
\begin{equation}
\begin{aligned}
&\sum_{x\in \mathbb{Z}_{4n^2+1}} (d(M(K_{D_k,n}), \mathfrak{s}_x)+d(M(K_{U,n}),\mathfrak{s}_{x}))\\
=&\sum_{x\in \mathbb{Z}_{4n^2+1}} (d(M(K_{D_k,n}), \mathfrak{s}_x)+d(M(K_{U,n}),\mathfrak{s}_{b_1x}))\geq 0.
\end{aligned}
\end{equation}
On the other hand from (4.3) we have 
\begin{equation}
\begin{aligned}
&\sum_{x\in \mathbb{Z}_{4n^2+1}} (d(M(K_{D_k,n}), \mathfrak{s}_x)+d(M(K_{U,n}),\mathfrak{s}_{x}))\\
=&\sum_{x\in \mathbb{Z}_{4n^2+1}} (d(M(K_{D_k,n}), \mathfrak{s}_x)+d(M(K_{U,n}),\mathfrak{s}_{b_2x}))\leq 0.
\end{aligned}
\end{equation}
From (4.4) and (4.5) we get 
\begin{equation}
\sum_{x\in \mathbb{Z}_{4n^2+1}} d(M(K_{D_k,n}), \mathfrak{s}_x)+d(M(K_{U,n}),\mathfrak{s}_{b_1x})=0.
\end{equation}
(4.6) and (4.2) then implies 
\begin{displaymath}
d(M(K_{D_k,n}), \mathfrak{s}_x)+d(M(K_{U,n}),\mathfrak{s}_{b_1x})= 0,\  \text{for all}\, x\in \mathbb{Z}_{4n^2+1}.
\end{displaymath}
This finishes the proof of the claim. 
\end{proof}
Now we are ready to specify a collection of parameters $k$ and $n$ so that the corresponding collection of knots $K_{k,n}$ generate $\mathbb{Z}_2^{\infty}$ in $\mathcal{T}_0/\mathcal{T}_1$. First we recall the following number theoretic result appeared in \cite{HKL16}, proved using a theorem due to Iwaniec \cite{Iwa78}.
\begin{prop}\label{nubertheory}
There exists an infinite set $\mathcal{N}$ of positive integers such that for all $n\in \mathcal{N}$, $4n^2+1$ is square free and is a product of at most two primes, furthermore, for all $n, m\in \mathcal{N}$ such that $m\neq n$, $4n^2+1$ and $4m^2+1$ are relatively prime.
\end{prop}

Take $\mathcal{N}$ to be a set as in the above proposition with the further requirement that $n>20$,  for all $n \in \mathcal{N}$. For each $n\in \mathcal{N}$, pick a $k=k(n)$ that depends on $n$ as in Proposition 4.1, then we form a family of knots $\mathcal{F}=\{K_{k,n}|\ n\in\mathcal{N},\  k=k(n)\}$ such that each member in $\mathcal{F}$ is 0-bipolar and yet not 1-bipolar.
\begin{proof}[Proof of Theorem 1.1]
Consider the subgroup in $\mathcal{T}_0$ generated by $\mathcal{F}$, we claim that it does not contain any element in $\mathcal{T}_1$. We prove this by contradiction, suppose $$L=K_{k_1,n_1}\#K_{k_2,n_2}\#\cdots\#K_{k_l,n_l}\in \mathcal{T}_1.$$ Without loss of generality, assume $n_i$'s are all distinct for $i=1,...,l$ since the knots are of order two. 
Since for distinct $m$ and $n$ in $\mathcal{N}$, $4n^2+1$ and $4m^2+1$ are relatively prime, each metabolizer on $H_1(M(L))$ splits as $G_1\oplus G_2$, where $G_1$ is a metabolizer on $H_1(M(K_{k_1,n_1}))$ and $G_2$ is a metabolizer on $H_1(M(K_{k_2,n_2}\#\cdots\#K_{k_l,n_l}))$. Consider the subgroup $G_1\oplus 0$ of metabolizer, then by Proposition \ref{metabolizer}, Theorem \ref{d invariant and 1-bipolar knots} and additivity of $d$-invariant, there exist $b_i$, $i=1,2$, such that $b_i^2+1\equiv 0 \	(\text{mod}\ 4n^2+1)$, and for all $x\in \mathbb{Z}_{4n_1^2+1}$,
\begin{displaymath}
\begin{aligned}
&d(M(K_{D_{k_1},n_1}), \mathfrak{s}_x)+d(M(K_{U,n_1}),\mathfrak{s}_{b_1x})+d(M(K_{k_2,n_2}\#...\#K_{k_l,n_l}),\mathfrak{s}_0)\\
=&d(M(K_{D_{k_1},n_1}), \mathfrak{s}_x)+d(M(K_{U,n_1}),\mathfrak{s}_{b_1x})\geq  0,
\end{aligned}
\end{displaymath}
and
\begin{displaymath}
\begin{aligned}
&d(M(K_{D_{k_1},n_1}), \mathfrak{s}_x)+d(M(K_{U,n_1}),\mathfrak{s}_{b_2x})+d(M(K_{k_2,n_2}\#...\#K_{k_l,n_l}),\mathfrak{s}_0)\\
=&d(M(K_{D_{k_1},n_1}), \mathfrak{s}_x)+d(M(K_{U,n_1}),\mathfrak{s}_{b_2x})\leq  0,
\end{aligned}
\end{displaymath}
where in the above equations we used the fact $d(M(K),\mathfrak{s}_0)=0$ when $K$ is a knot of smooth concordance order two \cite{MO07}.

The same argument in the proof of Proposition 4.1 shows $d(M(K_{D_{k_1},n_1}), \mathfrak{s}_x)+d(M(K_{U,n_1}),\mathfrak{s}_{b_1x})=0$, for all $x\in \mathbb{Z}_{4n_1^2+1}$, which contradicts our choice of $k_1$. Therefore, $L\notin \mathcal{T}_1$.
\end{proof}

\bibliographystyle{abbrv}\footnotesize
\bibliography{bipolar}
\end{document}